\documentclass[12pt]{amsart}

\usepackage{a4wide}
\usepackage{amsmath, amssymb}
\usepackage{amsthm}
\usepackage[utf8]{inputenc}
\usepackage{subfigure}
\usepackage{enumerate}
\usepackage{cancel}
\usepackage{xspace}
\usepackage{float}
\usepackage{mathrsfs}
\usepackage[colorlinks, citecolor=blue, anchorcolor=blue, linkcolor=blue, urlcolor=blue]{hyperref}
\usepackage{color,xcolor}

\usepackage{wrapfig}

\usepackage{tikz}
\usetikzlibrary{patterns}

\usepackage{algorithm}
\usepackage{algorithmic}

\def\SYoung#1{\vbox{\smallskip\offinterlineskip
    \halign{&\vbox{##}\kern-\SThickness\cr #1}}}

\newdimen\SSquaresize \SSquaresize=4.5pt
\newdimen\SThickness \SThickness=.15pt
\newdimen\SCorrection \SCorrection=7pt

\def\SCarre#1{\hbox{\vrule width \SThickness
   \vbox to \SSquaresize{\hrule height \SThickness\vss
      \hbox to \SSquaresize{\hss$\scriptstyle#1$\hss}
   \vss\hrule height\SThickness}
   \unskip\vrule width \SThickness}
   \kern-\SThickness}

\makeatletter \@addtoreset{equation}{section}

\newtheorem{theorem}{Theorem}[section]
\newtheorem{proposition}[theorem]{Proposition}

\newtheorem{corollary}[theorem]{Corollary}

\newtheorem{conjecture}[theorem]{Conjecture}

\newtheorem{claim}[theorem]{Claim}

\allowdisplaybreaks

\title[Extended reverse ultra log-concavity of transposed Boros-Moll sequences]{The extended reverse ultra log-concavity of transposed Boros-Moll sequences}

\author[James J. Y. Zhao]{James Jing Yu Zhao}
       \address{School of Accounting, Guangzhou College of Technology and Business,
       Foshan 528138, P.R. China.}
       \email{zhao@gzgs.edu.cn}

\begin{document}

\begin{abstract}
The Boros-Moll sequences $\{d_\ell(m)\}_{\ell=0}^m$ arise in the study of evaluation of a quartic integral. After the infinite log-concavity conjecture of the sequence $\{d_\ell(m)\}_{\ell=0}^m$ was proposed by Boros and Moll, a lot of interesting inequalities on $d_\ell(m)$  were obtained, although the conjecture is still open. Since $d_\ell(m)$ has two parameters, it is natural to consider the properties for the sequences $\{d_\ell(m)\}_{m\ge \ell}$, which are called the \emph{transposed Boros-Moll sequences} here. In this paper, we mainly prove the extended reverse ultra log-concavity of the transposed Boros-Moll sequences $\{d_\ell(m)\}_{m\ge \ell}$, and hence give an upper bound for the ratio ${d_\ell^2(m)}/{(d_\ell(m-1)d_\ell(m+1))}$. A lower bound for this ratio is also established which implies a result stronger than the log-concavity of the sequences $\{d_\ell(m)\}_{m\ge \ell}$. As a consequence, we also show that the transposed Boros-Moll sequences possess a stronger log-concave property than the Boros-Moll sequences do. At last, we propose some conjectures on the Boros-Moll sequences and their transposes.
\end{abstract}



\maketitle

\section{Introduction}

This paper is concerned with the extended reverse ultra log-concavity of the transposed Boros-Moll sequences $\{d_\ell(m)\}_{m\ge \ell}$.
Boros and Moll \cite{Boros-Moll-1999a, Boros-Moll-2001} investigated a quartic integral and provided a closed-form expression associated with a special class of Jacobi polynomials, that is,
$$
\int_0^\infty\frac{1}{(t^4+2xt^2+1)^{m+1}} dt
=\frac{\pi}{2^{m+3/2}(x+1)^{m+1/2}} P_m(x)
$$
for $x>-1$ and $m\in\mathbb{N}$, where the polynomial
\begin{align}\label{eq:B-M-poly1}
P_m(x)
=\sum_{j,k} \binom{2m+1}{2j}\binom{m-j}{k}\binom{2k+2j}{k+j}
 \frac{(x+1)^j(x-1)^k}{2^{3(k+j)}}.
\end{align}
By employing Ramanujan's Master Theorem, Boros and Moll proved that
\begin{align}\label{eq:B-M-poly2}
P_m(x)
=2^{-2m}\sum_{k=0}^m 2^k \binom{2m-2k}{m-k}\binom{m+k}{k} (x+1)^k,
\end{align}
which can be identified as the Jacobi polynomials $P_m^{(\alpha,\beta)}(x)$ with $\alpha=m+1/2$ and $\beta=-\alpha$, where
\begin{align*}
P_m^{(\alpha,\beta)}(x)
=\sum_{k=0}^m (-1)^{m-k} \binom{m+\beta}{m-k}\binom{m+k+\alpha+\beta}{k}
 \left(\frac{1+x}{2}\right)^k.
\end{align*}
Chen, Pang and Qu \cite{Chen-Pang-Qu} applied a combinatorial argument to show that the double sum \eqref{eq:B-M-poly1} can be reduced to the single sum \eqref{eq:B-M-poly2}.

The term $d_\ell(m)$ is the coefficient of $x^\ell$ in the polynomial $P_m(x)$, which is called the Boros-Moll polynomial, and the sequences $\{d_\ell(m)\}_{\ell=0}^m$ are called the Boros-Moll sequences. 
Clearly, one sees from \eqref{eq:B-M-poly2} that
\begin{align}\label{eq:B-M-seq}
d_\ell(m)=2^{-2m}\sum_{k=\ell}^m 2^k \binom{2m-2k}{m-k}\binom{m+k}{k}\binom{k}{\ell}
\end{align}
for $0\le \ell \le m$.
See \cite{Amdeberhan-Moll, Boros-Moll-1999b, Boros-Moll-1999c, Boros-Moll-2004, Moll2002} for more background on these sequences.

A sequence $\{a_i\}_{i\ge 0}$ with real numbers is said to be \emph{log-concave} if for any $i\ge 1$,
\begin{align}\label{eq:df-log-concave}
a_i^2-a_{i-1}a_{i+1}\ge 0.
\end{align}
If the inequality sign of \eqref{eq:df-log-concave} reverses, then the sequence $\{a_i\}_{i\ge 0}$ is called \emph{log-convex}. A polynomial is said to be log-concave if its coefficient sequence is log-concave, see Brenti \cite{Brenti} and Stanley \cite{Stanley}.

Boros and Moll \cite{Boros-Moll-1999b} showed that the sequence $\{d_\ell(m)\}_{\ell=0}^m$ is unimodal with the maximum term located in the middle, see also \cite{AMABKMR, Boros-Moll-1999c}.
Moll \cite{Moll2002} further conjectured that the sequences $\{d_\ell(m)\}_{\ell=0}^m$ are log-concave, which was proved by Kauers and Paule \cite{Kauers-Paule} with a computer algebra method. Chen {\it et al}. \cite{Chen-etal2012}  also gave a combinatorial proof for this conjecture by building a structure of partially $2$-colored permutations.

A sequence $\{a_i\}_{i=0}^n$ is called ultra log-concave if $\{a_i/\binom{n}{i}\}_{i=0}^n$ is log-concave, that is,
\begin{align}\label{def:ulcc}
 \frac{a_i^2}{\binom{n}{i}^2}
\ge \frac{a_{i-1}}{\binom{n}{i-1}}\cdot \frac{a_{i+1}}{\binom{n}{i+1}},
\end{align}
see Liggett \cite{Liggett}.
Clearly, the inequality \eqref{def:ulcc} implies
\begin{align*}
i(n-i)a_i^2-(n-i+1)(i+1)a_{i-1}a_{i+1}\ge 0,
\end{align*}
which is stronger than \eqref{eq:df-log-concave}. It is well-known that the coefficients of a realrooted polynomial form an ultra log-concave sequence. Liggett \cite{Liggett} also mentioned that the ultra log-concavity of a sequence $\{a_i\}_{i=0}^n$ implies the log-concavity of the sequence $\{i!a_i\}_{i=0}^n$.

A sequence $\{a_i\}_{i=0}^n$ is called reverse ultra log-concave if the reverse relation in \eqref{def:ulcc} holds. 
For instance, Han and Seo \cite{HanSeo} showed the log-concavity and reverse ultra log-concavity of the Bessel polynomial
\begin{align*}
B_n(x)=\sum_{k=0}^n \frac{(n+k)!}{2^k k! (n-k)!} x^k.
\end{align*}
Moreover, Chen and Gu \cite[Theorems 1.1 \& 1.2]{Chen-Gu2009} proved that for $m\ge 2$ and $1\le \ell \le m-1$,
\begin{align}\label{eq:ub-mell2}
 \frac{d_\ell^2(m)}{d_{\ell-1}(m)d_{\ell+1}(m)}
<\frac{(m-\ell+1)(\ell+1)}{(m-\ell)\ell},
\end{align}
and
\begin{align}\label{eq:lb-mell2}
 \frac{d_\ell^2(m)}{d_{\ell-1}(m)d_{\ell+1}(m)}
>\frac{(m-\ell+1)(\ell+1)(m+\ell)}{(m-\ell)\ell(m+\ell+1)}.
\end{align}
Clearly, the inequality \eqref{eq:ub-mell2} implies the reverse ultra log-concavity of the Boros-Moll sequences. And \eqref{eq:lb-mell2} is stronger than the log-concavity of the Boros-Moll sequences.
Their results suggest that, in the asymptotic sense, the Boros-Moll sequences are just on the borderline between ultra log-concavity and reverse ultra log-concavity.

The Boros-Moll sequences $\{d_\ell(m)\}_{\ell=0}^m$ satisfy many other interesting inequalities.
For instance, Chen and Xia \cite[Theorem 1.1]{Chen-Xia2009} showed that the Boros-Moll polynomials possess the strictly ratio monotone property, which implies both log-concavity and the spiral property.
Chen, Wang and Xia \cite{Chen-Wang-Xia} introduced the notion of interlacing log-concavity of a sequence of polynomials with positive coefficients which is stronger than the log-concavity of the polynomials themselves, and showed the interlacing log-concavity of $\{P_m(x)\}_{m\ge 0}$.

For a sequence $\{a_i\}_{i\ge 0}$ of real numbers, define an operator $\mathcal{L}$ by $\mathcal{L}(\{a_i\}_{i\ge 0})=\{b_i\}_{i\ge 0}$, where $b_i=a_i^2-a_{i-1}a_{i+1}$ for $i\ge 0$, with the convention that $a_{-1}=0$. Boros and Moll \cite{Boros-Moll-2004} introduced the notion of infinite log-concavity.
A sequence $\{a_i\}_{i\ge 0}$ is said to be $k$-log-concave if the sequence $\mathcal{L}^j\left(\{a_i\}_{i\ge 0}\right)$ is nonnegative for each $1\le j\le k$, and $\{a_i\}_{i\ge 0}$ is said to be $\infty$-log-concave if $\mathcal{L}^k\left(\{a_i\}_{i\ge 0}\right)$ is nonnegative for any $k\ge 1$. The following conjecture was proposed by
Boros and Moll and is still open.
\begin{conjecture}\cite{Boros-Moll-2004}\label{conj:B-M-inflv}
The Boros-Moll sequence $\{d_\ell(m)\}_{\ell=0}^m$ is $\infty$-log-concave.
\end{conjecture}
Br\"{a}nd\'{e}n \cite{Branden} provided an approach to Conjecture \ref{conj:B-M-inflv} by relating real-rooted polynomials to higher-order log-concavity. Although, as shown by Boros and Moll \cite{Boros-Moll-1999b}, the polynomials $P_m(x)$ are not real-rooted in general, Br\"{a}nd\'{e}n introduced two polynomials derived from $P_m(x)$ and conjectured the real-rootedness of them \cite[Conjectures 8.5 \& 8.6]{Branden}, which have been confirmed by Chen, Dou and Yang \cite{Chen-Dou-Yang}, and hence the $2$-log-concavity and the $3$-log-concavity of the BorosMoll polynomials were obtained.
In another direction, Chen and Xia \cite{Chen-Xia2} showed a proof of the $2$-log-concavity of the Boros-Moll sequences by using the approach of recurrence relations.

Guo \cite{Guo2022} proved the higher order Tur\'{a}n inequalities of the Boros-Moll sequences by showing an equivalent form \cite[Eq. (9)]{Guo2022} established in \cite{Guo2021}. Zhao \cite{Zhao2023} gave a simple proof of these higher order Tur\'{a}n inequalities by employing a sufficient condition built by Hou and Li \cite[Theorem 5.2]{Hou-Li2021}, together with a set of sharp enough bounds of $d_\ell^2(m)/(d_{\ell-1}(m)d_{\ell+1}(m))$ given in \eqref{eq:ub-mell2} and  \cite[Theorem 3.1]{Zhao2023}.

Since $d_\ell(m)$ has two parameters, it is natural to investigate properties for the sequences $\{d_\ell(m)\}_{m\ge \ell}$, which are called the \emph{transposed Boros-Moll sequences} in this paper.

The sequences $\{d_\ell(m)\}_{m\ge \ell}$ were proved to be log-convex for $\ell=0$, log-concave for $\ell\ge 1$ and $2$-log-concave for $\ell\ge 2$ by Jiang and Wang \cite{Jiang-Wang2024}. The higher order Tur\'{a}n inequalities for the sequences $\{d_\ell(m)\}_{m\ge \ell}$ were  also derived for $\ell\ge 2$ in \cite{Jiang-Wang2024}.

Recently, Zhang and Zhao \cite{Zhang-Zhao} showed that the Boros-Moll sequences $\{d_\ell(m)\}_{\ell=0}^m$, its normalizations $\{d_\ell(m)/\ell!\}_{\ell=0}^m$, and its transposes $\{d_\ell(m)\}_{m\ge \ell}$ satisfy the Briggs inequality, which arising from Briggs' conjecture that if a polynomial $a_0+a_1x+\cdots+a_nx^n$ with real coefficients has only negative zeros, then
$a^2_k(a^2_k - a_{k-1}a_{k+1}) > a^2_{k-1}(a^2_{k+1} - a_ka_{k+2})$
for any $1\leq k\leq n-1$. In order to prove the Briggs inequality for the sequence $\{d_\ell(m)\}_{m\ge \ell}$, they established the strict ratio-log-convexity of  $\{d_\ell(m)\}_{m\ge \ell}$ for $\ell\ge 1$. As a consequence, the strict log-convexity of the sequence $\{\sqrt[n]{d_\ell(\ell+n)}\}_{n\ge 1}$ for $\ell\ge 1$ was also obtained.

In this paper, we mainly show that the transposed Boros-Moll sequences $\{d_\ell(m)\}_{m\ge \ell}$ possess the extended reverse ultra log-concavity property.
A sequence $\{a_i\}_{i\ge k}$ is called extended ultra log-concave if $\{a_i/\binom{i}{k}\}_{i\ge k}$ is log-concave, and the extended reverse ultra log-concavity of the sequence $\{a_i\}_{i\ge k}$ is defined in a similar way of the reverse ultra log-concavity.

The remainder of this paper is organized as follows. In Section \ref{Sec:Recu}, we first recall some known recurrence relations for $d_\ell(m)$ which will be applied in our proofs. In Section \ref{Sec:erulcv}, we first prove the extended reverse ultra log-concavity of the transposed Boros-Moll sequences $\{d_\ell(m)\}_{m\ge \ell}$, and hence give an upper bound for the ratio $d_\ell^2(m)/(d_\ell(m-1)d_\ell(m+1))$. We further establish a lower bound for this ratio in Theorem \ref{thm:lbdlm2}, which implies an inequality stronger than the log-concavity of the sequences $\{d_\ell(m)\}_{m\ge \ell}$.
As will be seen, the upper and lower bounds for $d_\ell^2(m)/(d_\ell(m-1)d_\ell(m+1))$ are very close to each other, it may be said that, in the asymptotic sense, the sequences $\{d_\ell(m)\}_{m\ge \ell}$ are just on the borderline between extended ultra log-concavity and extended reverse ultra log-concavity for any $\ell\ge 1$. Finally, we propose some conjectures on the Boros-Moll sequences and their transposes in Section \ref{Sec:Conj}.

\section{The recurrences}\label{Sec:Recu}

Kauers and Paule \cite{Kauers-Paule} used a computer algebra system to derive the following recurrence relations for $d_\ell(m)$, which will be employed in our proofs.
For $\ell\ge 0$ and $m\ge \ell$,
\begin{align}
 4(m^2+m)(m+1-\ell)d_\ell(m+1)
=&\ 2m(8m^2+8m-4\ell^2+3)d_\ell(m)\label{eq:rec-dlm-m} \\
 &\quad -(16m^2-1)(m+\ell)d_\ell(m-1), \nonumber \\
 (m+2-\ell)(m+\ell-1)d_{\ell-2}(m)
=&\ (2m+1)(\ell-1)d_{\ell-1}(m)
  -\ell(\ell-1)d_\ell(m),\label{eq:rec-i-ho}\\
 2(m+1)d_\ell(m+1)
=&\ 2(m+\ell)d_{\ell-1}(m)+(4m+2\ell+3)d_\ell(m),\label{eq:reclm1}\\
 2(m+1)(m+1-\ell)d_\ell(m+1)
=&\ (4m-2\ell+3)(m+\ell+1)d_{\ell}(m)\label{eq:reclm2}\\
 &\quad -2\ell(\ell+1)d_{\ell+1}(m).\nonumber
\end{align}

It should be mentioned that Moll \cite{Moll2007} independently derived the relations \eqref{eq:rec-dlm-m} and \eqref{eq:rec-i-ho} via the WZ-method \cite{WZ}. As remarked by Chen and Xia \cite[Sec. 2]{Chen-Xia2009}, the recursions \eqref{eq:reclm1} and \eqref{eq:reclm2} can be easily deduced from \eqref{eq:rec-dlm-m} and \eqref{eq:rec-i-ho}, moreover, \eqref{eq:rec-dlm-m} and \eqref{eq:rec-i-ho} can be also derived from \eqref{eq:reclm1} and \eqref{eq:reclm2}.

\section{The main results}\label{Sec:erulcv}

The objective of this section is to show the main result of this paper, the extended reverse ultra log-concavity of the transposed Boros-Moll sequences $\{d_\ell(m)\}_{m\ge \ell}$.

\begin{theorem}\label{thm:rulcm}
For each $\ell\ge 0$, the transposed Boros-Moll sequence $\{d_\ell(m)\}_{m\ge \ell}$ is strictly extended reverse ultra log-concave.
That is, for each $\ell\ge 1$ and $m\ge \ell+1$, we have
\begin{align}\label{ineq:rulcm}
\left(\frac{d_\ell(m)}{\binom{m}{\ell}}\right)^2
< \left(\frac{d_\ell(m-1)}{\binom{m-1}{\ell}}\right)\cdot
  \left(\frac{d_\ell(m+1)}{\binom{m+1}{\ell}}\right)
\end{align}
or, equivalently,
\begin{align}\label{ineq:ubdlm2}
\frac{d_\ell^2(m)}{d_\ell(m-1)d_\ell(m+1)}
< \frac{(m-\ell+1)m}{(m-\ell)(m+1)}.
\end{align}
\end{theorem}

We further establish a lower bound for $d_\ell^2(m)/(d_\ell(m-1)d_\ell(m+1))$, which implies an inequality stronger than the log-concavity of the transposed Boros-Moll sequences $\{d_\ell(m)\}_{m\ge \ell}$.
\begin{theorem}\label{thm:lbdlm2}
For each $\ell\ge 0$ and $m\ge \ell+1$, we have
\begin{align}\label{ineq:lbdlm2}
\frac{d_\ell^2(m)}{d_\ell(m-1)d_\ell(m+1)}
> \frac{(m-\ell+1)m^3}{(m-\ell)(m+1)(m^2+1)}.
\end{align}
\end{theorem}
It is easily checked that for $\ell\ge 2$ and $m\ge \ell+1$,
\begin{align}\label{ineq:1.10}
\frac{(m-\ell+1)m^3}{(m-\ell)(m+1)(m^2+1)}>\frac{m^2+1}{m^2}.
\end{align}
Consequently, we obtain the following relation from Theorem \ref{thm:lbdlm2}.
\begin{corollary}\label{coro:LcEtBMs}
For each $\ell\ge 2$ and $m\ge \ell+1$, we have
\begin{align*}
\frac{d_\ell^2(m)}{d_\ell(m-1)d_\ell(m+1)}>\frac{m^2+1}{m^2}.
\end{align*}
\end{corollary}
Clearly, Theorems \ref{thm:rulcm} and \ref{thm:lbdlm2} imply, respectively, that the transposed Boros-Moll sequence $\{d_\ell(m)\}_{m\ge \ell}$ is strictly log-convex for $\ell=0$ and is strictly log-concave for each $\ell\ge 1$.
Besides, Corollary \ref{coro:LcEtBMs} establishes an inequality which is stronger than the log-concavity of the transposed Boros-Moll sequences.

Moreover, we obtain the following relation which implies that the transposed Boros-Moll sequences possess a stronger log-concave property than the Boros-Moll sequences do.
\begin{proposition}\label{Prop:3.5}
For $\ell\ge 1$ and $m\ge \ell+1$, we have
\begin{align}\label{ieq:Lgvm>Lgvell}
 d_\ell^2(m)>d_\ell(m-1)d_\ell(m+1)>d_{\ell-1}(m)d_{\ell+1}(m).
\end{align}
\end{proposition}

\begin{proof}
Fixed $\ell\ge 1$ and $m\ge \ell+1$. The first inequality in \eqref{ieq:Lgvm>Lgvell} was proved in \cite[Theorem 3.1]{Jiang-Wang2024}, which can also be derived from Theorem \ref{thm:lbdlm2}. Combining \eqref{ineq:ubdlm2} and \eqref{eq:lb-mell2}, it follows that
\begin{align*}
\frac{d_\ell^2(m)}{d_\ell(m-1)d_\ell(m+1)}
<\frac{(m-\ell+1)m}{(m-\ell)(m+1)}
<\frac{(m-\ell+1)(\ell+1)(m+\ell)}{(m-\ell)\ell(m+\ell+1)}
<\frac{d_\ell^2(m)}{d_{\ell-1}(m)d_{\ell+1}(m)},
\end{align*}
which yields the second inequality in \eqref{ieq:Lgvm>Lgvell}.
\end{proof}

\subsection{A lower bound for $d_{\ell+1}(m)/d_\ell(m)$}\label{Sec:lb2}

In order to prove Theorem \ref{thm:rulcm}, we first establish a sufficiently sharp lower bound for the ratio $d_{\ell+1}(m)/d_\ell(m)$ which is stated in Theorem \ref{thm:dl1mm>}.
For $\ell\ge 1$ and $m\ge \ell+1$, set
\begin{align}\label{def:Lml}
W(\ell,m)=\frac{m(2m+1)(2\ell+3)-\sqrt{\Delta_1}}{4m(\ell^2+\ell)},
\end{align}
where
\begin{align}\label{def:Delta_1}
\Delta_1=52m^4+(64\ell^2+56)m^3+(16\ell^4+36\ell^2+13)m^2-8\ell^2m-4\ell^2.
\end{align}

\begin{theorem}\label{thm:dl1mm>}
Let $W(\ell,m)$ be given by \eqref{def:Lml}. For integers $\ell\ge 1$ and $m\ge \ell+1$, we have
\begin{align}\label{ineq:dl1mm>}
\frac{d_{\ell+1}(m)}{d_\ell(m)}> W(\ell,m).
\end{align}
\end{theorem}

\begin{proof}
Note that Theorem \ref{thm:dl1mm>} is equivalent to the following statement. That is,
\begin{align}\label{ineq:dl1mm>rsm}
\frac{d_{\ell+1}(m)}{d_\ell(m)}> W(\ell,m),
\end{align}
for $m\ge 2$ and $1\le \ell \le m-1$.
So, we aim to prove \eqref{ineq:dl1mm>rsm} by using induction on $m$. For $m=2$ and $\ell=1$, it is easy to check that
\begin{align*}
\frac{d_{2}(2)}{d_1(2)}-W(1,2)
=\frac{2}{5}-\frac{25-2\sqrt{127}}{8}
=\frac{10\sqrt{127}-109}{40}>0.
\end{align*}
Assume that \eqref{ineq:dl1mm>rsm} is true, that is, for $1\le \ell \le m-1$,
\begin{align}\label{ineq:eqdl1mm>}
d_{\ell+1}(m)> W(\ell,m) d_\ell(m).
\end{align}
It suffices to prove that for $1\le \ell \le m$,
\begin{align}\label{ineq:eqdl1mm1>}
d_{\ell+1}(m+1)> W(\ell,m+1) d_\ell(m+1).
\end{align}

For $\ell=m$, we have $d_{m+1}(m+1)/d_m(m+1)=2/(2m+3)$, and
$$
W(m,m+1)=\frac{4m^3+16m^2+21m+9-\sqrt{\omega}}{4m(m+1)^2},
$$
where
$$
\omega=16m^6+96m^5+296m^4+520m^3+581m^2+402m+121>0.
$$
Direct computation gives that
\begin{align*}
\frac{d_{m+1}(m+1)}{d_m(m+1)}-W(m,m+1)
=\frac{(2m+3)\sqrt{\omega}-(8m^4+36m^3+74m^2+73m+27)}{4m(2m+3)(m+1)^2}>0,
\end{align*}
since $(2m+3)^2 \omega-(8m^4+36m^3+74m^2+73m+27)^2=$$4(4m+3)(4m+5)(m^2+6m+6)>0$.
Thus, \eqref{ineq:eqdl1mm1>} holds for $\ell=m$.

It remains to show \eqref{ineq:eqdl1mm1>} for $1\le \ell \le m-1$. To this end, applying the recurrence relations \eqref{eq:reclm1} and \eqref{eq:reclm2},
we have
\begin{align*}
d_{\ell+1}(m+1)
=&\ \frac{m+\ell+1}{m+1}d_{\ell}(m)+\frac{4m+2\ell+5}{2(m+1)}d_{\ell+1}(m), \\
d_\ell(m+1)
=&\ \frac{(4m-2\ell+3)(m+\ell+1)}{2(m+1)(m+1-\ell)}d_{\ell}(m)
 -\frac{\ell(\ell+1)}{(m+1)(m+1-\ell)}d_{\ell+1}(m),
\end{align*}
for $1\le \ell \le m-1$. 
Then the inequality \eqref{ineq:eqdl1mm1>} can be rewritten as
\begin{align}\label{ineq:dl1mm1>re}
P \cdot d_{\ell+1}(m)> Q\cdot d_\ell(m),
\end{align}
where
\begin{align*}
P=&\ 8m^3+32m^2+43m-4\ell^2m-4\ell^2+19-\sqrt{\Delta_2}, \\
Q=&\ \frac{(m+1+\ell)\left((m+1)F-(4m-2\ell+3)\sqrt{\Delta_2}\right)}{2(\ell^2+\ell)},
\end{align*}
with
\begin{align*}
\Delta_2
=&\ 52m^4+(64\ell^2+264)m^3+(16\ell^4+228\ell^2+493)m^2+(32\ell^4+256\ell^2+402)m\\
 &\quad +16\ell^4+88\ell^2+121,\\
F=&\ 16\ell m^2+24m^2-16\ell^2 m+16\ell m+54m+8\ell^3-12\ell^2-8\ell+27.
\end{align*}
Clearly, $\Delta_2>0$ and $F>0$. Observe that $P>0$, because for $1\le \ell \le m-1$, we have $8m^3+32m^2+43m-4\ell^2m-4\ell^2+19>0$ and
\begin{align*}
 &(8m^3+32m^2+43m-4\ell^2m-4\ell^2+19)^2-\Delta_2\\
=&\ 64m^4(m^2-\ell^2)+128m^3(4m^2-3\ell^2)+4m^2(415m^2-207\ell^2)
 +8m(349m^2-94\ell^2)\\
 &\ +(2572m^2-240\ell^2)+1232m+240>0.
\end{align*}
Thus, in view of \eqref{ineq:eqdl1mm>} and \eqref{ineq:dl1mm1>re}, it is sufficient to show that for $1\le \ell \le m-1$,
\begin{align}\label{ineq:sfdl1mm>}
P \cdot W(\ell,m)> Q.
\end{align}

With the aid of a computer, it is easy to check that
\begin{align*}
P \cdot W(\ell,m)-Q
=&\ \frac{G_1+G_2\sqrt{\Delta_2}-(G_3-\sqrt{\Delta_2})\sqrt{\Delta_1}}{4m(\ell^2+\ell)},
\end{align*}
where
\begin{align*}
G_1=&\ (m^2+m)(12m^2+24m-16\ell^4+28\ell^2+3),\\
G_2=&\ 8m^3+8m^2-4\ell^2m+3m,\\
G_3=&\ (m+1)(8m^2+24m-4\ell^2+19).
\end{align*}
Observer that $G_3>G_2>0$ for $1\le \ell \le m-1$. Moreover, $G_3-\sqrt{\Delta_2}>0$ since
\begin{align*}
G_3^2-\Delta_2=&\ 4(4m+3)(4m+5)(m+2)^2(m+1+\ell)(m+1-\ell)>0.
\end{align*}
So we have $(G_3-\sqrt{\Delta_2})\sqrt{\Delta_1}>0$. To prove \eqref{ineq:sfdl1mm>}, we need to determine the sing of $G_1+G_2\sqrt{\Delta_2}$.
\begin{claim}\label{cla:G12>0}
For $m\ge 2$ and $1\le \ell \le m-1$, we have $G_1+G_2\sqrt{\Delta_2}>0$.
\end{claim}

Since $G_1=(m^2+m)(12m^2+24m-16\ell^4+28\ell^2+3)$, it is clear that for any given $\ell\ge 1$, $G_1\ge 0$ for sufficiently large $m$.
If $G_1\ge 0$, then Claim \ref{cla:G12>0} holds. We proceed to prove the case that $G_1<0$. In this case, we have $G_1+G_2\sqrt{\Delta_2}=G_2\sqrt{\Delta_2}-|G_1|$.
Notice that
\begin{align*}
G_2^2 {\Delta_2}-|G_1|^2
=&\ 4m^2(4m+3)(4m+5)(m+1+\ell)(m+1-\ell)(52m^4+64\ell^2m^3+160m^3\\
 &\ +16\ell^4m^2+100\ell^2m^2+161m^2-8\ell^2m+80m-32\ell^4-20\ell^2+18)>0,
\end{align*}
which leads to $G_1+G_2\sqrt{\Delta_2}>0$. Thus Claim \ref{cla:G12>0} is proved.

We proceed to show
\begin{align}\label{ineq:G123Delta12}
G_1+G_2\sqrt{\Delta_2}>(G_3-\sqrt{\Delta_2})\sqrt{\Delta_1}.
\end{align}
In order to do so, we derive that
\begin{align*}
(G_1+G_2\sqrt{\Delta_2})^2-(G_3-\sqrt{\Delta_2})^2 {\Delta_1}
= -H_1+H_2\sqrt{\Delta_2},
\end{align*}
where
\begin{align*}
H_1
=&\ 8(m+1)^2(832m^7+1536\ell^2m^6+4576m^6+512\ell^4m^5+7104\ell^2m^5+9556m^5\\
 &\ +1792\ell^4m^4+11648\ell^2m^4+9358m^4+2048\ell^4m^3+7588\ell^2m^3+4192m^3\\
 &\ +800\ell^4m^2+770\ell^2m^2+646m^2-32\ell^6m+120\ell^4m-1018\ell^2m\\
 &\ -16\ell^6+32\ell^4-241\ell^2),\\
H_2
=&\ 8(m+1)(128m^6+128\ell^2m^5+496m^5+448\ell^2m^4+672m^4+480\ell^2m^3+368m^3\\
 &\ +120\ell^2m^2+64m^2+8\ell^4m-62\ell^2m+4\ell^4-19\ell^2).
\end{align*}
Clearly, $H_1>0$ and $H_2>0$ for $1\le \ell \le m-1$. In view of
\begin{align*}
H_2^2 \Delta_2-H_1^2
=&\ 256(m+1)^2(4m+3)^2(4m+5)^2\left((m+1)^2-\ell^2\right)^2(156m^8+192\ell^2m^7\\
 &\ +636m^7+816\ell^2m^6+891m^6+256\ell^4m^5+1256\ell^2m^5+498m^5\\
 &\ +64\ell^6m^4+720\ell^4m^4+812\ell^2m^4+87m^4+128\ell^6m^3+576\ell^4m^3\\
 &\ +208\ell^2m^3+48\ell^6m^2+88\ell^4m^2+19\ell^2m^2-64\ell^4m-16\ell^4)>0,
\end{align*}
we deduce that $H_2\sqrt{\Delta_2}-H_1>0$, which leads to \eqref{ineq:G123Delta12}, as well as \eqref{ineq:sfdl1mm>}, for $1\le \ell \le m-1$.
This completes the proof.
\end{proof}

\subsection{Proof of Theorem \ref{thm:rulcm}}

We are now in a position to prove Theorem \ref{thm:rulcm}.
\begin{proof}[Proof of Theorem \ref{thm:rulcm}.]
Applying the recurrence relations \eqref{eq:rec-dlm-m} and \eqref{eq:reclm2}, the inequality \eqref{ineq:ubdlm2} can be rewritten as
\begin{align}\label{ineq:qua<0}
A \left(\frac{d_{\ell+1}(m)}{d_\ell(m)}\right)^2
+B \left(\frac{d_{\ell+1}(m)}{d_\ell(m)}\right)+C<0,\quad m\ge \ell+1,
\end{align}
where
\begin{align*}
&A=4m^2\ell^2(\ell+1)^2,\\
&B=-2m^2(2m+1)\ell(\ell+1)(2\ell+3),\\
&C=(m-\ell)[4(\ell^2+3\ell-1)m^3+(4\ell^3+8\ell-5)m^2-(2\ell+1)m-\ell].
\end{align*}
The discriminant of the above quadratic function on $d_{\ell+1}(m)/d_\ell(m)$ is
\begin{align*}
\Delta
=4m^2\ell^2(\ell+1)^2 \Delta_1,
\end{align*}
where $\Delta_1$ is given by \eqref{def:Delta_1}.
Clearly, $\Delta_1>0$, and hence $\Delta>0$ for $\ell\ge 1$ and $m\ge \ell+1$. Thus, the quadratic function in \eqref{ineq:qua<0} has two distinct zeros, that is,
\begin{align*}
x_1=\frac{m(2m+1)(2\ell+3)-\sqrt{\Delta_1}}{4m(\ell^2+\ell)},\\
x_2=\frac{m(2m+1)(2\ell+3)+\sqrt{\Delta_1}}{4m(\ell^2+\ell)}.
\end{align*}
Since $A>0$, it suffices to prove that for $\ell\ge 1$ and $m\ge \ell+1$,
\begin{align}
x_1<\frac{d_{\ell+1}(m)}{d_\ell(m)}<x_2.
\end{align}

By Chen and Xia \cite[Lemma 3.1]{Chen-Xia2009}, for $\ell\ge 1$ and $m\ge \ell+1$,
\begin{align}
\frac{d_{\ell+1}(m)}{d_\ell(m)}<\frac{m-\ell}{\ell+1}.
\end{align}
It is easy to check that
\begin{align*}
\frac{m-\ell}{\ell+1}-x_2
=-\frac{m(4\ell^2+2\ell+6m+3)+\sqrt{\Delta_1}}{4m(\ell^2+\ell)}<0,
\end{align*}
which leads to $d_{\ell+1}(m)/d_\ell(m)<x_2$ for $\ell\ge 1$ and $m\ge \ell+1$.
It remains to show that for $\ell\ge 1$ and $m\ge \ell+1$,
\begin{align}
\frac{d_{\ell+1}(m)}{d_\ell(m)}> x_1,
\end{align}
which is obtained in Theorem \ref{thm:dl1mm>}.
This completes the proof.
\end{proof}

\subsection{Proof of Theorem \ref{thm:lbdlm2}}\label{Sec:pfThm2.2}

The goal of this part is to complete the proof of Theorem \ref{thm:lbdlm2}, the lower bound for $d_\ell^2(m)/(d_\ell(m-1)d_\ell(m+1))$.

\begin{proof}[Proof of Theorem \ref{thm:lbdlm2}]
Fix $\ell\ge 0$.
Applying the recurrence relations \eqref{eq:rec-dlm-m}, the inequality \eqref{ineq:lbdlm2} can be rewritten as
\begin{align}\label{ineq:qua>0}
\mathcal{A} \left(\frac{d_{\ell}(m+1)}{d_\ell(m)}\right)^2
+\mathcal{B} \left(\frac{d_{\ell}(m+1)}{d_\ell(m)}\right)
+\mathcal{C}>0,\quad m\ge \ell+1,
\end{align}
where
\begin{align*}
&\mathcal{A}=4m^4(m+1-\ell)^2(m+1),\\
&\mathcal{B}=-2m^4(m+1-\ell)(8m^2+8m-4\ell^2+3),\\
&\mathcal{C}=(16m^2-1)(m+1)(m^2+1)(m^2-\ell^2).
\end{align*}
The discriminant of the above quadratic function on $d_{\ell}(m+1)/d_\ell(m)$ is
\begin{align*}
\Delta=4m^4(m+1-\ell)^2\cdot f_\ell(m),
\end{align*}
where
\begin{align}
f_\ell(m)
=&\ -12m^6+(64\ell^2-72)m^5+(16\ell^4+100\ell^2-47)m^4+(120\ell^2+8)m^3 \label{defi:fellm}\\
 &\ +(56\ell^2+4)m^2-8\ell^2m-4\ell^2.\nonumber
\end{align}

We aim to prove \eqref{ineq:qua>0}. For this purpose, we need to determine the sign of $\Delta$ first.
For $\ell=0$, we have
\begin{align*}
f_0(m)=-m^2[(m-1)(12m^3+84m^2+131m+123)+119]<0
\end{align*}
for $m\ge 1$. Thus $\Delta<0$ for $\ell=0$ and $m\ge 1$. So \eqref{ineq:qua>0} is proved for $\ell=0$, since $\mathcal{A}>0$.

We proceed to prove \eqref{ineq:qua>0} for $\ell\ge 1$.
Given $\ell\ge 1$, observe that the number of sign changes of the coefficients sequence of $f_\ell(m)$ is $2$. Thus by Descartes' rule of signs, the polynomial $f_\ell(m)$ has at most two positive zeros. In view of the fact that $f_\ell(0)=-4 \ell^2<0$, $f_\ell(1)=16\ell^4+328\ell^2-119>0$ and $f_\ell(+\infty)=-\infty$, the polynomial $f_\ell(m)$ has only two positive zeros, denoted by $r_1(\ell)$ and $r_2(\ell)$, where $0<r_1(\ell)<1<r_2(\ell)<+\infty$.

Clearly, $f_\ell(\ell+1)=16\ell^8+128\ell^7+504\ell^6+1080\ell^5+1085\ell^4+44\ell^3
-826\ell^2-588\ell-119>0$ for any integer $\ell\ge 1$. Thus, $r_2(\ell)>\ell+1$. Then we shall discuss in two cases.

{\bf Case 1}.
$m> \lfloor r_2(\ell)\rfloor\ge \ell+1$.
In this case, we have $f_\ell(m)<0$, and hence $\Delta<0$, which yields \eqref{ineq:qua>0} since $\mathcal{A}>0$.

{\bf Case 2}.
$r_1(\ell)<\ell+1\le m \le \lfloor r_2(\ell)\rfloor$. In this case, we have $f_\ell(m)\ge 0$, and hence $\Delta\ge 0$.
Thus, the quadratic function in \eqref{ineq:qua>0} has two zeros, that is,
\begin{align*}
 x_\ell(m)
=\frac{m^2(8m^2+8m-4\ell^2+3)-\sqrt{f_\ell(m)}}{4m^2(m+1)(m+1-\ell)},\\
 y_\ell(m)
=\frac{m^2(8m^2+8m-4\ell^2+3)+\sqrt{f_\ell(m)}}{4m^2(m+1)(m+1-\ell)}.
\end{align*}

To prove \eqref{ineq:qua>0} for $\ell\ge 1$, we have the following claim.
\begin{claim}\label{cl:dl1mlm>}
For $\ell\ge 1$ and $\ell+1\le m \le \lfloor r_2(\ell)\rfloor$, we have
\begin{align}\label{ineq:dl1m1>ylm}
\frac{d_{\ell}(m+1)}{d_\ell(m)}>y_\ell(m).
\end{align}
\end{claim}

Note that Zhao \cite[Theorem 2.1]{Zhao2023} obtained a lower bound for $d_{\ell}(m+1)/d_\ell(m)$. That is, for any $m\ge 2$ and $1\le \ell \le m-1$,
\begin{align}\label{ineq:lb-dlm1m}
\frac{d_\ell(m+1)}{d_\ell(m)}>U(\ell,m),
\end{align}
where
\begin{align}\label{eq:def-Lml}
 U(\ell,m)
=\frac{(4m^2+7m-2\ell^2+3)(m+\ell^2)
  +\ell\sqrt{\lambda_\ell(m)}}{2(m+1)(m-\ell+1)(m+\ell^2)},
\end{align}
with
\begin{align}\label{defi:labalm}
\lambda_\ell(m)=(m+\ell^2)(4\ell^4+8\ell^2m+5\ell^2+m).
\end{align}
Clearly, we have \eqref{ineq:lb-dlm1m} holds for $\ell\ge 1$ and $m\ge \ell+1$. Then it is sufficient to show that for $\ell\ge 1$ and $\ell+1\le m \le \lfloor r_2(\ell)\rfloor$,
\begin{align}\label{ineq:Lml>ylm}
U(\ell,m)\ge y_\ell(m).
\end{align}

We proceed to prove \eqref{ineq:Lml>ylm}.
By a simple computation, we have
\begin{align*}
 U(\ell,m)-y_\ell(m)
=\frac{K_1+K_2\sqrt{\lambda_\ell(m)}-K_3\sqrt{f_\ell(m)}}
{4m^2(m+1)(m-\ell+1)(m+\ell^2)},
\end{align*}
where
\begin{align*}
&K_1=3m^2(2m+1)(m+\ell^2),\\
&K_2=2\ell m^2,\\
&K_3=m+\ell^2.
\end{align*}
Thus, it remains to show
\begin{align}\label{ineq:K123}
K_1+K_2\sqrt{\lambda_\ell(m)}-K_3\sqrt{f_\ell(m)}\ge 0.
\end{align}
Note that $K_1,K_2,K_3>0$ for $\ell\ge 1$ and $m\ge \ell+1$. It follows that
\begin{align*}
&\left(K_1+K_2\sqrt{\lambda_\ell(m)}\right)^2
 -\left(K_3\sqrt{f_\ell(m)}\right)^2\\
=&\ 4(m+\ell^2)\left(m^2 M_1+(2\ell^4+\ell^2)m+\ell^4+3\ell m^4 (2m+1)\sqrt{\lambda_\ell(m)}
            -m^2 M_2\right)\\
>&\ 4m^2(m+\ell^2)\left(M_1+3\ell m^2(2m+1)\sqrt{\lambda_\ell(m)}
            -M_2\right),
\end{align*}
where
\begin{align*}
&M_1=12m^5+27m^4+(3\ell^2+14)m^3+\ell^2,\\
&M_2=4\ell^2m^4+12\ell^4m^3+(20\ell^4+16\ell^2+2)m^2+(30\ell^4+16\ell^2+1)m
       +14\ell^4.
\end{align*}
Obviously, $M_1,M_2>0$ for $\ell\ge 1$ and $m\ge \ell+1$.
We proceed to prove
\begin{align*}
M_1+3\ell m^2(2m+1)\sqrt{\lambda_\ell(m)}-M_2>0.
\end{align*}
To this end, we derive that
\begin{align*}
\left(M_1+3\ell m^2(2m+1)\sqrt{\lambda_\ell(m)}\right)^2-M_2^2
=S_\ell(m)+T_\ell(m)\sqrt{\lambda_\ell(m)},
\end{align*}
where
\begin{align*}
 S_\ell(m)
=&\ 144m^{10}+648m^9+(272\ell^4+108\ell^2+1065)m^8+(336\ell^6+504\ell^4
 +198\ell^2+756)m^7\\
&\ +(452\ell^6+169\ell^4+77\ell^2+196)m^6-(336\ell^8+336\ell^6+122\ell^4
 -16\ell^2)m^5\\
&\ -(1084\ell^8+1091\ell^6+360\ell^4+10\ell^2+4)m^4-(1536\ell^8+1600\ell^6
 +666\ell^4\\
&\ +68\ell^2+4)m^3-(1460\ell^8+1408\ell^6+372\ell^4+32\ell^2+1)m^2\\
&\ -(840\ell^8+448\ell^6+28\ell^4)m-(196\ell^8-\ell^4),\\
T_\ell(m)=&\ 6\ell m^2(2m+1)\big(12m^5+27m^4+(3\ell^2+14)m^3+\ell^2\big).
\end{align*}
Observe that $T_\ell(m)>0$. We claim that $S_\ell(m)>0$ for $\ell\ge 1$ and $m\ge \ell+1$.
Note that
\begin{align*}
 S_\ell(\ell)
=&\ -\ell^2(360\ell^{10}+1368\ell^9+2130\ell^8+1716\ell^7+822\ell^6
   +342\ell^5+186\ell^4+96\ell^3\\
 &\qquad\quad +35\ell^2+4\ell+1)<0,\\
 S_\ell(\ell+1)
=&\ 312\ell^{12}+2880\ell^{11}+11402\ell^{10}+26950\ell^9+48379\ell^8
   +84450\ell^7+146585\ell^6\\
 &\ +211096\ell^5+223433\ell^4+164196\ell^3
  +78696\ell^2+22230\ell+2800>0,
\end{align*}
for $\ell\ge 1$. 
Thus there exists a real $x_0\in(\ell,\ell+1)$ such that $S_\ell(x_0)=0$ for any $\ell\in \mathbb{R}$. It is clear that the number of sign changes of the coefficients sequence of $S_\ell(m)$ is one for $\ell\ge 1$. Therefore, by Descartes' rule of signs, the polynomial $S_\ell(m)$ has only one positive zero $x_0$ for $\ell\ge 1$. It follows that $S_\ell(m)>0$ for  $\ell\ge 1$ and $m\ge \ell+1$. So, we have
\begin{align*}
S_\ell(m)+T_\ell(m)\sqrt{\lambda_\ell(m)}>0
\end{align*}
for $\ell\ge 1$ and $m\ge \ell+1$, which leads to \eqref{ineq:K123}, and hence  \eqref{ineq:Lml>ylm} is proved.
\end{proof}

Recall that \eqref{eq:ub-mell2} and \eqref{eq:lb-mell2} give two bounds for $d_\ell^2(m)/(d_{\ell-1}(m)d_{\ell+1}(m))$ which were established by Chen and Gu \cite[Theorems 1.1 \& 1.2]{Chen-Gu2009} while studying the reverse ultra log-concavity of the Boros-Moll polynomials. 
The distance between these two bounds is
\begin{align*}
 D_1
=\frac{(m-\ell+1)(\ell+1)}{(m-\ell)\ell(m+\ell+1)}.
\end{align*}
These two bounds are very close to each other since $D_1$ is very small. As mentioned by Chen and Gu \cite{Chen-Gu2009}, in the asymptotic sense, the Boros-Moll sequences are just on the borderline between ultra log-concavity and reverse ultra log-concavity.

Also notice that the distance between the upper and lower bounds for $d_\ell^2(m)/(d_\ell(m-1)d_\ell(m+1))$ provided by Theorems \ref{thm:rulcm} and \ref{thm:lbdlm2}, respectively, is
\begin{align*}
 D_2
=\frac{(m-\ell+1)m}{(m-\ell)(m+1)(m^2+1)}.
\end{align*}
Clearly, for any given $\ell\ge 1$, $\lim_{m\rightarrow\infty} D_2/m^{-2}=1$, which implies that $D_2$ is very small. It means that the two bounds given in Theorems \ref{thm:rulcm} and \ref{thm:lbdlm2} are very sharp. Therefore, in the asymptotic sense, we may say that the transposed Boros-Moll sequences are just on the borderline between extended ultra log-concavity and extended reverse ultra log-concavity.

\section{Some conjectures}\label{Sec:Conj}
To conclude this paper, we propose some conjectures related to the Boros-Moll sequences and their transposes.

Motivated by Boros and Moll's $\infty$-log-concavity conjecture, we first propose a conjecture on the $\infty$-log-concavity of the transposed Boros-Moll sequences $\{d_\ell(m)\}_{m\geq \ell}$.
Note that the sequences $\{d_0(m)\}_{m\ge 0}$ and $\mathcal{L}\big(\{d_1(m)\}_{m\geq 1}\big)$ were proved to be log-convex by Jiang and Wang \cite{Jiang-Wang2024}. For $\ell=2$, it is easily checked with a computer that the sequence $\mathcal{L}^5\big(\{d_2(m)\}_{m\geq 2}\big)$ got many negative terms, which implies that the sequence $\{d_2(m)\}_{m\geq 2}$ is not $5$-log-concave. However, for $\ell\geq 3$, numerical computation reveals that the beginning terms of the sequences $\mathcal{L}^j\big(\{d_\ell(m)\}_{m\geq \ell}\big)$ are positive and increase very fast, which suggests the following conjecture.

\begin{conjecture}
The transposed Boros-Moll sequences $\{d_\ell(m)\}_{m\geq \ell}$ are $\infty$-strictly-log-concave for any $\ell\geq 3$.
\end{conjecture}

The next two conjectures are on log-concavity and extended reverse ultra log-concavity of the sequences $\{d_\ell(m-1)d_\ell(m+1)/d_\ell^2(m)\}_{m\ge \ell+1}$.
\begin{conjecture}
For each $\ell\ge 1$, the sequence $\{d_\ell(m-1)d_\ell(m+1)/d_\ell^2(m)\}_{m\ge \ell+1}$ is log-concave.
\end{conjecture}

\begin{conjecture}
For each $\ell\ge 0$, the sequence $\{d_\ell(m-1)d_\ell(m+1)/d_\ell^2(m)\}_{m\ge \ell+1}$ is extended reverse ultra log-concave.
\end{conjecture}

Numerical experiments suggest that the ratio sequence $\{d_\ell(m)/d_{\ell-1}(m)\}_{1\le \ell\le m}$ is neither log-concave nor log-convex. This sequence may have a distinctive log-behavior with half log-convex and half log-concave.

\begin{conjecture}\label{conj:hlc-hlv}
For $m\ge 3$, the sequence $\{d_\ell(m)/d_{\ell-1}(m)\}_{1\le \ell \le \lfloor m/2 \rfloor+1}$ is log-convex, and the sequence $\{d_\ell(m)/d_{\ell-1}(m)\}_{\lfloor m/2 \rfloor\le \ell \le m}$ is log-concave. That is,
\begin{align}\label{ineq:rlcv-l}
\left(\frac{d_\ell(m)}{d_{\ell-1}(m)}\right)^2
< \left(\frac{d_{\ell-1}(m)}{d_{\ell-2}(m)}\right)
  \left(\frac{d_{\ell+1}(m)}{d_\ell(m)}\right)
\end{align}
holds for $2\le \ell \le \lfloor m/2 \rfloor$, and the inequalities in \eqref{ineq:rlcv-l}
reverse for $\lfloor m/2 \rfloor +1\le \ell \le m-1$.
\end{conjecture}

Clearly, Conjecture \ref{conj:hlc-hlv} is equivalent to that $r_\ell(m)<1$ for $2\le \ell \le \lfloor m/2 \rfloor$ and $r_\ell(m)>1$ for $\lfloor m/2 \rfloor +1\le \ell \le m-1$, where
\begin{align}
r_\ell(m)=\frac{d_\ell^3(m) d_{\ell-2}(m)}{d_{\ell-1}^3(m) d_{\ell+1}(m)}.
\end{align}

We have verified Conjecture \ref{conj:hlc-hlv} for $3\le m\le 200$.
For example, $\left\{r_\ell(9)\right\}_{2\le \ell \le 8}$ is given by
\begin{align*}
&r_2(9)=\frac{60275815334620606439322}{78173355142115765635889},\qquad r_3(9)=\frac{122118613523526671413768}{133528261319822227027923},\\[7pt] &r_4(9)=\frac{135495563425805832093}{139776208550739676384},\qquad
r_5(9)=\frac{2512968603767684}{2503881674347833},\ \
\\[7pt]
&r_6(9)=\frac{3844942434909}{3698150303624},\qquad r_7(9)=\frac{2672864807424}{2420889681239},\qquad
r_8(9)=\frac{3879265207}{2951578112}.
\end{align*}
We see that
$$
r_2(9)<1,\ \ r_3(9)<1,\ \ r_4(9)<1,\ \ r_5(9)>1,\ \ r_6(9)>1,\ \ r_7(9)>1,\ \ r_8(9)>1.
$$

Chen, Guo and Wang \cite{CGW2014} introduced the notion of \emph{infinitely log-monotonic}. For a sequence $\{a_i\}_{i\ge 0}$ of real numbers, define an operator $\mathcal{R}$ by $\mathcal{R}\{a_i\}_{i\ge 0}=\{b_i\}_{i\ge 0}$, where $b_i=a_{i+1}/a_i$ for $i\ge 0$. A sequence $\{a_i\}_{i\ge 0}$ is called log-monotonic of order $k$ if for $j$ odd and not exceeds $k-1$, the sequence $\mathcal{R}^j\{a_i\}_{i\ge 0}$ is log-concave and for $j$ even and not exceeds $k-1$, the sequence $\mathcal{R}^j\{a_i\}_{i\ge 0}$ is log-convex. A sequence $\{a_i\}_{i\ge 0}$ is called \emph{infinitely log-monotonic} if it is log-monotonic of order $k$ for all integers $k\ge 1$. By applying the log-behavior of the Riemann zeta function and the gamma function, Chen {\it et al}. \cite{CGW2014} also showed the infinite log-monotonicity of the Bernoulli numbers, the Catalan numbers and the central binomial coefficients.

The transposed Boros-Moll sequence $\{d_0(m)\}_{m\ge 0}$ was proved to be log-convex in \cite[Theorem 3.1]{Jiang-Wang2024} and is easily checked to be ratio log-concave, and hence is log-monotonic of order $2$. Further numerical experiments suggests the following conjecture.

\begin{conjecture}
The transposed Boros-Moll sequence $\{d_\ell(m)\}_{m\ge \ell}$ is infinitely log-monotonic for $\ell=0$. The sequence $\{d_\ell(m+1)/d_\ell(m)\}_{m\ge \ell}$ is infinite log-monotonic for each $\ell\ge 1$.
\end{conjecture}

\section*{Acknowledgments}
The author would like to thank Arthur L. B. Yang for valuable suggestion and comments for previous version.

\end{document}